\documentclass[a4paper, 11pt]{article}

\usepackage[
            includefoot,  %Uncomment to put page number above margin
            marginparwidth=0in,     % Length of section titles
            marginparsep=0in,       % Space between titles and text
            margin=1.1in,               % 1 inch margins
            includemp]{geometry}
\usepackage[utf8x]{inputenc}
\usepackage{bbm}
\usepackage{float}
\usepackage{graphicx}                  % derni\`ere \'etant la langue principale
\usepackage{amssymb}
\usepackage{amsfonts}
\RequirePackage{amsmath}
\RequirePackage{amsthm}

\usepackage{color}
\usepackage[colorlinks,linkcolor=blue,citecolor=blue,urlcolor=blue]{hyperref}

\usepackage{enumerate}

\usepackage[english]{babel}
\newtheorem{thm}{Theorem}[section]
\newtheorem{defn}[thm]{Definition}
\newtheorem{prop}[thm]{Proposition}
\newtheorem{lem}[thm]{Lemma}

\newtheorem{cor}[thm]{Corollary}

\newtheorem{rmq}{Remark}[section]

\DeclareMathOperator{\Hess}{Hess}

\newcommand{\R}{\mathbb{R}}
\newcommand{\N}{\mathbb{N}}

\newcommand{\cL}{\mathcal{L}}

\begin{document}

\title{Stability estimates for the sharp spectral gap bound under a curvature-dimension condition}
\date{\today}
\author{Max Fathi, Ivan Gentil and Jordan Serres}

\maketitle

\begin{abstract}
We study stability of the sharp spectral gap bounds for metric-measure spaces satisfying a curvature bound. Our main result, new even in the smooth setting, is a sharp quantitative estimate showing that if the spectral gap of an RCD$(N-1, N)$ space is almost minimal, then the pushforward of the measure by an eigenfunction associated with the spectral gap is close to a Beta distribution. The proof combines estimates on the eigenfunction obtained via a new $L^1$-functional inequality for RCD spaces with Stein's method for distribution approximation. We also derive analogous, almost sharp, estimates for infinite and negative values of the dimension parameter. 
\end{abstract}

\section{Introduction}

Our goal in this work is to study the stability of sharp spectral gap bounds for Markov diffusion operators $\cL$ satisfying a curvature lower bound. Formally (and rigorously in the smooth setting), these operators satisfy both the Bochner inequality, (or curvature-dimension  condition CD$(K,N)$),  
\begin{equation} \label{def_cd}
\frac{1}{2}\cL \Gamma(f) - \Gamma(f, \cL f) \geq K\Gamma(f) + \frac{1}{N}(\cL f)^2 
\end{equation}
for all smooth functions, and the diffusion property: for any smooth bounded function $\phi$ 
\begin{equation}
\label{77}
\cL \phi(f) = \phi'(f)\cL f + \phi''(f)\Gamma(f).
\end{equation}
This definition has been proposed in the seminal paper~\cite{BE85} by the way of the operator $\Gamma_2(f) :=\frac{1}{2}\cL \Gamma(f) - \Gamma(f, \cL f) $.

An alternative definition of the curvature-dimension condition (now known as the Lott-Sturm-Villani theory~\cite{LV09,St06}) is to require geodesic convexity properties of the entropy in the space of probability measures. In the setting of smooth Riemanian manifolds endowed with a reference probability measure, the curvature-dimension condition is equivalent to a lower bound on the weighted Ricci curvature tensor. However, in the non-smooth setting, requiring the inequality~\eqref{def_cd} to hold pointwise is slightly too strong. We shall be interested in so-called RCD spaces, that satisfy the diffusion property, as well as a weak, integral form of the curvature-dimension condition~\cite{AGS15, EKS15}. This definition is equivalent to the more classical definition via convexity properties of the entropy functional as in the Lott-Sturm-Villani theory when the space is infinitesimally Hilbertian. 

For a smooth ($\mathcal C^2$ and complete)  $N$-dimensional Riemannian manifold endowed with its Laplacian, if the Ricci curvature tensor is bounded from below by $N-1$ (which implies the CD$(N-1,N)$ condition), we have a lower bound on the spectral gap, or first positive eigenvalue of the Laplacian: 
\begin{equation} \label{lichn}
\lambda_1(-\Delta) \geq N.
\end{equation}
This bound is sharp, since equality holds for the $N$-sphere of radius $1$. Moreover, this bound is \emph{rigid}: equality holds iff the manifold is isometric to the $N$-sphere~\cite{Oba62}. This phenomenon is related to rigidity results for other sharp bounds, such as on the diameter or the volume. 

The next question is whether the bound is stable: if a manifold with Ricci curvature bounded from below by $N-1$ has a spectral gap close to $N$, is it close in some sense to the $N$-sphere? The answer to this question is negative~\cite{And90}. To get closeness in Gromov-Hausdorff distance, Aubry~\cite{Aub05} showed that we must ask that the $N$-th eigenvalue is close to $N$, which is true for the sphere, since the multiplicity of $N$ as an eigenvalue is $N+1$, and that $\lambda_{N-1}$ being close to one is not enough. This improved on an earlier result of Petersen~\cite{Pet99} who showed that $\lambda_{N+1} \approx N$ suffices. The reason for this phenomenon is that smooth manifolds are not quite the right setting for the problem: it is possible to extend the notion of Ricci curvature lower bounds to non-smooth weighted manifolds, and in that setting there are spaces other than the $N$-sphere for which equality holds in~\eqref{lichn}, namely spherical suspensions. We refer to~\cite{Ket15} for their description. 

When it is only $\lambda_1$ that is close to one, Cheng~\cite{Che75} showed that the diameter is close to $\pi$, and Croke~\cite{Cro82} proved the converse statement, still in the smooth, unweighted setting. Later, Bertrand~\cite{Ber07} showed that $\lambda_k \approx N$ for $k \leq n$ implies that the manifold contains a piece that is close to $\mathbb{S}^k$. 

We shall now discuss the literature in the non-smooth setting of RCD spaces. This curvature condition, that we shall present in details in Section~\ref{prelim_rcd}, extends Ricci curvature bounds, and can be introduced using either the Bakry-Ledoux gradient estimates, a weak form of the Bochner inequality or convexity properties of the entropy along geodesics. This setting also makes sense when $N$ is not an integer. The sharp spectral gap estimate for RCD$(N-1,N)$ spaces was proved in~\cite{EKS15}, and cases of equality were fully described in~\cite{Ket15, Ket15b}. More recently, Cavaletti and Mondino tackled rigidity and stability results for geometric comparison theorems in the non-smooth setting using the so-called needle decomposition, a technique pioneered by Klartag~\cite{Kla17} to reduce such problems to one-dimensional statements. Most relevant to the present work is the quantitative Obata theorem of~\cite{CMS19}, which states that on an essentially non-branching CD$(N-1,N)$ space $(M, d, \mu)$ with spectral gap $\lambda_1$ and associated normalized eigenfunction $u$ there exists a point $x_0 \in M$ such that
$$
||u - \sqrt{N+1}\cos(d(\cdot, x_0))||_2 \leq C(N)(\lambda_1 - N)^{1/(8N+4)}; \hspace{2mm} \pi - \operatorname{diam}(M) \leq C(N)(\lambda_1 - N)^{1/N}.
$$
A variant of the diameter estimate was obtained in~\cite{JZ16} in the RCD setting, still with a dimension-dependent exponent. We shall include a comparison between the technical estimates on the eigenfunction used in~\cite{CMS19} and those used here in Section~\ref{sect_eigen_est}. The main differences between Theorem \ref{main_thm_intro} below and the results of~\cite{CMS19}, beyond the norms used being different, is that on the upside we get a dimension-independent exponent in our main estimate, but with the downside of requiring the RCD condition rather than the more general CD condition. Topological sphere theorems were also considered in the RCD setting, in~\cite{HM21}, and an averaged version of the maximal diameter theorem was proved in \cite{ES21}, where the only optimizer is the $N$-sphere. 

Our main result is a sharp quantitative estimate on how far the distribution of the pushforward measure by an eigenfunction is of being a beta distribution. More precisely, we show that

\begin{thm} \label{main_thm_intro}
Let $(M, d , \mu)$ be an RCD$(N-1,N)$ space with $N > 1$, unit mass and spectral gap $\lambda_1 \leq N + \varepsilon$ for some $\varepsilon > 0$. Let $f$ be an eigenfunction of the Laplacian, with eigenvalue $\lambda_1$, and normalized so that $||\Gamma(f)||_1 = N/(N+1)$. There is a constant $C(N) > 0$ (independent of $M$ and $f$) such that the $L^1$-Wasserstein distance between the pushforward of $\mu$ by $f$ and a symmetrized Beta distribution with parameters $(N/2, N/2)$ is smaller than $C(N)\varepsilon$. 
\end{thm}

The order of magnitude $\varepsilon$ in the bound is \emph{sharp}, as can be checked by considering an $N$-sphere of radius $1-\varepsilon$. The choice of the value of the normalization for $||\Gamma(f)||_1$ is to match the value for coordinates on a unit sphere in dimension $N$, and without the pushforward would simply be close to a scaled Beta distribution. 

Symmetrized Beta distributions (with parameters $(N/2, N/2)$) have densities proportional to $(1-x^2)^{N/2-1}$ on $[-1, 1]$. They appear in this statement because they are precisely the distribution of coordinates on unit spheres, so our statement can be viewed as saying that the space contains a piece that is close to a piece of the $N$-sphere (when the parameter $N$ is an integer of course). 

Our method follows an approach developed in~\cite{BF21} for proving stability of the sharp spectral gap estimate on RCD$(1, \infty)$ spaces, where the model space is the Gauss space. It combines quantitative estimates on eigenfunctions with Stein's method for comparing probability distributions via approximate integration by parts formulas. We shall actually sharpen the quantitative bound of order $\varepsilon^{1/2}$ of~\cite[Theorem 1.3]{BF21} in the RCD$(1,\infty)$ setting to $\varepsilon \log(1/\varepsilon)$, which is within a logarithmic factor of being sharp. 

An important tool in our analysis is a new $L^1$-functional inequality (Proposition~\ref{prop-35}). It arises as a limit case of a family of $L^p$-functional inequalities introduced by Meyer~\cite{Mey84} and more recently revisited in~\cite{MN14,EI21} in the Gaussian setting for $p > 1$. As we shall see, the $L^1$-inequality fails in the Gaussian setting, but holds for RCD$(N-1,N)$ spaces. Another new result of independent interest (Theorem~\ref{thm_beta_eigen}) is a general criterion for proving that an eigenfunction has a distribution close to a Beta distribution, which is a variant of a result of E. Meckes~\cite{Mec09} for the Gaussian setting. 

As noted in~\cite{Sch03,Oht16, KM18, bgs20}, the definition of the Bakry-Emery condition also makes sense for \emph{negative} values of $N$. Sharp functional inequalities and model spaces in this setting were studied for example in~\cite{Mil17}. Rigidity in the smooth setting was studied in~\cite{Mai19}. Under this curvature condition, the manifold may have infinite volume, so we shall require finite volume as an extra condition. Then the one-dimensional model space is $\R$ endowed with a non-constant metric and a generalized Cauchy distribution~\cite{GZ21}. We shall derive a stability estimate similar to that of the RCD$(1,\infty)$ case, under an additional integrability condition on the eigenfunction, as well as assuming smoothness. To do so, we derive a version of Stein's lemma for one-dimensional generalized Cauchy distributions, of independent interest. We shall also give a Cauchy counterpart to E. Meckes' theorem on Gaussianity of eigenfunctions we previously mentioned.

The sequel is as follows: in Section~\ref{prelim_rcd}, we shall present background results on RCD spaces, and a description of the model spaces for both positive and negative values of the dimension parameter. Section~\ref{sec-3} shall contain the proof of our main Theorem, while Sections~\ref{sec-4} and~\ref{sec-5} shall respectively contain the results on the infinite-dimensional case and the negative-dimensional case.    

\section{Preliminaries on the setting and curvature-dimension condition}
\label{prelim_rcd}

We briefly explain in next sections what we call {\it smooth space} and {\it metric measure space}.  
\subsection{Some generalities on smooth setting}
\label{sec-2.3}

We briefly describe here the smooth setting, a complete description can be found in~\cite[Sec.~3.2]{BGL14}.

The smooth setting that we consider in this paper, other possibilities are available in the literature, is a smooth weighted manifold. Let $(M,g)$ be a  ($\mathcal C^\infty$) connected and complete $d$-dimensional Riemannian manifold ($d\geq 1$).  Let $W$ be a smooth function on $M$ and let $\mu=e^{-W}Vol$ be the reference measure where $Vol$ is the Riemannian measure. We assume in this article that $\mu$ is a probability measure and it is the case under the curvature-dimension condition used here. 

The generator is described on smooth functions $f\in\mathcal C_c^\infty(M)$ by
$$
\cL(f)=\Delta_g f-\Gamma(W,f),
$$
where $\Delta_g$ is the Laplace-Beltrami operator and $\Gamma$ is the associated carr\'e du champ operator. The Markov semigroup with generator $\cL$ is noted $(P_t)_{t\geq 0}$. For any smooth functions $f,h$, we have 
$$
\Gamma(f,h)=\langle\nabla f,\nabla h\rangle_g=\nabla f\cdot \nabla h,
$$
that is scalar product associated with the metric $g$. 

The weighted manifold $(M,g)$ associated with the measure $\mu$ satisfies a CD$(K,N)$ condition whenever the inequality~\eqref{def_cd} is satisfied for all smooth function $f$. Following for instance~\cite[Sec.~C.6]{BGL14}, if $N\in\R\setminus[0,d]$ then the curvature-dimension condition CD$(K,N)$ is equivalent to the following inequality on tensors, 
\begin{equation}
\label{58}
Ric_g-\nabla\nabla^gW\geq K g+\frac{1}{N-d}\nabla W\otimes\nabla W,
\end{equation}
where $Ric_g$ is the Ricci tensor of $(M,g)$ and $\nabla\nabla^gW$ is the Hessian of $W$ with respect to the metric $g$.

\subsection{Some generalities on metric measure spaces}
\label{sec-2.1}

We consider a complete, separable, metric measure space $(M, d, \mu)$, where $\mu$ is a probability measure. We can define the Cheeger energy of an $L^2$-function $f$ as 
$$
\operatorname{Ch}(f) := \frac{1}{2} \inf_{(f_i)_{i\in\N}}\liminf_{i\rightarrow\infty} \int{(\operatorname{Lip} f_i)^2d\mu} 
$$
where the infimum runs over all sequences of locally-lipschitz functions converging to $f$ in $L^2$, and $\operatorname{Lip} f(x)$ is the local lipschitz constant at $x$. If an $L^2$-function $f$ has finite Cheeger energy, then there exists a minimal weak upper gradient, which we shall denote as $\Gamma(f)^{1/2}$, such that
$$\operatorname{Ch}(f) = \frac{1}{2}\int{\Gamma(f)d\mu}.$$
The Sobolev space $W^{1,2}(M,d,\mu)$ is the space of $L^2$-functions with finite Cheeger energy. We refer to~\cite{AGS14,AGS15} for more about these notions. For smooth functions on a Riemanian manifold, $\Gamma(f)$ coincides with $|\nabla f|^2$. 

The space is said to be infinitesimally Hilbertian if the Cheeger energy is quadratic, that is
$$\operatorname{Ch}(f+g) + \operatorname{Ch}(f-g) = 2\operatorname{Ch}(f) + \operatorname{Ch}(g)$$
for all $f, g \in W^{1,2}$. Smooth manifolds are of course infinitesimally Hilbertian spaces, but there are examples of spaces, such as Finsler spaces, that do not satisfy this condition. 

If the space is infinitesimally Hilbertian, we can define the scalar product $\Gamma(f,g)=\langle \nabla f, \nabla g\rangle \in L^1$ of two elements of $W^{1,2}$ by polarization of $\Gamma$, as well as the Dirichlet form $\mathcal{E}(f,g) = \int{\Gamma(f,g)d\mu}$. The analogue of the Laplace operator for the space $(M, d, \mu)$ is then the operator $\cL : D(\cL) \longrightarrow L^2$ such that
$$
\mathcal{E}(f,g) = -\int{g(\cL f) d\mu}.
$$
The domain $D(\cL)$ of the operator is dense in $L^2$.  The associated Markov semigroup is also noted $(P_t)_{t\geq 0}$. We refer to~\cite{AGS14b} for background about this construction.  We can then define RCD spaces as follows: 

\begin{defn}[RCD$(K,N)$ spaces]
	\label{def-21}
A complete, separable, metric-measure space $(M, d, \mu)$ is said to be an RCD$(K, N)$ space with $K \in \R$ and $N \in \R \backslash [0, 1]$ if it is infinitesimally Hilbertian and if it satisfies the Bochner inequality in a weak form: for any $f \in D(\cL)$ with $\cL f \in W^{1,2}(M, d, \mu)$ and $g \in D(\cL) \cap L^\infty(\mu)$ with $g \geq 0$ and $\cL g \in L^\infty$ we have
\begin{equation}
\label{24}
\frac{1}{2}\int{ \cL g \Gamma(f)d\mu} - \int{g\Gamma(f,\cL f) d\mu} \geq K\int{g\Gamma(f)d\mu} + \frac{1}{N}\int{g(\cL f)^2d\mu}.
\end{equation}
\end{defn}
The main issue differentiating this weak formulation of the pointwise Bochner inequality~\eqref{def_cd} is that $\cL \Gamma(f)$ might not be well-defined. Examples of RCD spaces include classical smooth manifolds satisfying Ricci curvature bounds, but also their possible Gromov-Hausdorff limits, as well as certain stratified spaces that do not arise as limits of smooth manifolds~\cite{BKMR21}. When $N > 1$, this definition is equivalent to the Lott-Sturm-Villani definition of Ricci curvature lower bounds, up to the extra assumption of linearity of the heat flow. 

The now classical bound on the spectral gap is the following, proved for example in~\cite[Theorem 4.22]{EKS15} in the RCD setting: 
\begin{thm}[Spectral gap of $\cL$]
	\label{thm-22}
Let $(M,d,\mu)$ a RCD$(\rho, N)$ metric measure space with $\rho > 0$ and $N \in \R \backslash [0, 1]$, the first eigenvalue $\lambda_1>0$ of $-\cL$ satisfies
$$
\lambda_1 \geq \frac{N\rho}{N-1}.
$$ 
\end{thm}

Moreover, existence of an eigenfunction associated with the spectral gap was proved in~\cite{GMS15} under the RCD$(\rho,\infty)$ condition with $\rho>0$. 
%Almost all results proved in the smooth setting are also true on metric measures spaces.  

\subsection{Model spaces in dimension 1}
\label{sec-22}
We describe here the two main model spaces in dimension 1 used in the paper. The idea is to briefly describe how they are defined.

Let us consider $\varphi$, a smooth and positive function on an open interval $I\subset  \R$. And let define the generator 
$$
\cL f=\varphi f''-(\beta-1)\varphi'f',
$$
for smooth function $f$.
This generator has a reversible  measure $\varphi^{-\beta}dx$, for any smooth and compactly supported functions on $I$, 
$$
\int_If\cL g\varphi^{-\beta}dx=-\int_I \varphi f'g'\varphi^{-\beta}dx.
$$
The carr\'e du champ operator $\Gamma$, defined by 
$$
\Gamma(f)=\frac12 \cL(f^2)-f\cL f,
$$
is given by  $\Gamma(f)=\varphi \times (f')^2$. It says that we are working on the open interval $I$ associated with the metric $1/\varphi$. From this new metric, the generator $\cL$ takes the form
$$
\cL f=\Delta_\varphi f-\Gamma(W,f)
$$
where $\Delta_\varphi f=\varphi f''+\frac{\varphi'}{2}f'$ is the Laplace-Beltrami operator on the one dimensional manifold $(I,1/\varphi)$ and 
$W=(\beta-\frac{1}{2})\log\varphi$.

Following~\cite[Sec.~C.6]{BGL14}, here in dimension 1, the operator $\cL$ satisfies the curvature-dimension condition  $CD(\rho,N)$ (with $N\notin[0,1]$) if and only if
\begin{equation}
\label{10}
\nabla\nabla^ {\varphi}W-\frac{\rho}{\varphi}\geq \frac{(W')^2}{N-1},
\end{equation}
where $\nabla\nabla^ {\varphi}W=W''+W'\frac{\varphi'}{2\varphi}$
is the Hessian of $W$ with respect to the metric $1/\varphi$.   Equation~\eqref{10} becomes, 
\begin{equation}
\label{11}
\left(\beta-\frac{1}{2}\right)\left[\varphi''-\frac{\rho}{\beta-\frac{1}{2}}-\frac{\varphi'^2}{\varphi}\frac{N+2(\beta-1)}{2(N-1)}\right]\geq0.
\end{equation}
The two main examples, used in the paper, are when there is equality in~\eqref{11}. 
\begin{enumerate}
	\item Let $I=(-1,1)$ and $\varphi=1-x^2$,  $\rho=1-2\beta$ and $N=2(1-\beta)$, ($\beta<1/2$). Then the generator
$$
\cL^+ f=(1-x^2)f''-Nxf'
$$	
is the so-called  Jacobi operator and satisfies the curvature-dimension condition $CD(N-1,N)$ with $N>1$. The carré du champ is 
$
\Gamma(f)=(1-x^2)(f')^2
$
and the reversible measure is a Beta distribution,
\begin{equation}
\label{78}
d\mu^+_N=\frac{(1-x^2)^{N/2-1}}{Z^+}1_{[-1,1]}dx,
\end{equation}
where $Z^+$ is such that $\mu^+_N$ it a probability measure. This measure is the so-called symmetrized  Beta distribution with parameters $(N/2, N/2)$.

The first non-trivial eigenvalue  for the operator $-\cL^+$ is $N$,  with the  eigenfunction $f(x)=x$. Moreover we have 
$$
Var_{\mu^+_N}(f)=\int x^2d\mu_N^+-\left(\int xd\mu_N^+\right)^2=\int x^2d\mu_N^+=\frac{1}{N+1}
$$ 
while $\int{\Gamma(f)d\mu_N^+} =  \frac{N}{N+1}$. 

When $N\geq 2$ is a positive integer, $\mu_N^+$ is the distribution of a coordinate on the $N$-dimensional unit sphere $\mathbf{S}^N$. 

\item Let now $I=\R$, $\varphi=1+x^2$, $\rho=2\beta-1$ and $N=2(1-\beta)$. We assume that $\beta>3/2$, so that $N<-1$.  For this model, we have   
$$
\cL^- f=(1+x^2)f''+Nxf'
$$	
and $\Gamma(f)=(1+x^2)f'^2$, 
and the reversible measure is a Cauchy type distribution (also called Student distribution),  
\begin{equation}
d\mu^-_N=\frac{(1+x^2)^{N/2-1}}{Z^-}dx,
\end{equation}
where $Z^-$ is such that $\mu^-_N$ it is a probability measure.  This model satisfies the curvature-dimension condition $CD(1-N,N)$ with $N<-1$.  

Again, the first non trivial eigenvalue for the operator $\cL^-$ is $N$  with the eigenfunction is $f(x)=x$. We also have $Var_{\mu^-_N}(f)=\frac{-1}{N+1}$ and $\int{\Gamma(f)d\mu_N^-} =  \frac{N}{N+1}$.  
 
The assumption $N<-1$ is a necessary condition for the variance of $f$ to exist. 
\end{enumerate} 

\begin{rmq}
Of course, these models could be parametrized differently, without effect on the curvature or the spectral gap. In particular, our model for negative $N$ is the same as the model space in~\cite{Mil17}, except that we parametrize it differently. It is important to notice that we chose to deal with these parametrizations of the two models (in dimension 1) so that the eigenfunction associated with the first eigenvalue  is the identity function, and satisfies $\int{\Gamma(f)d\mu_N^\pm} =  \frac{N}{N+1}$.
\end{rmq}

\subsection{$L^1$-Wasserstein distance}
\label{sec-27}
The $W_1$ distance, also called the $L^1$-Wasserstein distance is the optimal transport distance in $L^1$. For any $\mu,\nu$ probability measure on a metric space $(M,d)$, 
$$
W_1(\mu, \nu)=\inf\int d(x,y)d\pi(x,y),
$$
where the infimum  is running over all probability measures $\pi$ on $M\times M$ which admit $\mu$ and $\nu$ as marginals.  From the Kantorovich-Rubinstein's theorem (see~\cite{AGS14} for instance) we have 
$$
W_1(\mu, \nu) = \sup_{||g||_{Lip} \leq 1} \Big\{\int{g d\mu} - \int{g d\nu}\Big\},
$$
and this is the definition used in this paper.

\section{The positive and finite-dimensional case}
\label{sec-3}

Let $N>1$ be a real number and  $(M,d,\mu)$ be a RCD$(N-1,N)$ metric measure space as proposed in Definition~\ref{def-21}. As is classical, we consider this case rather than more general RCD$(\rho,N)$ spaces since we can normalize the value of $\rho$ to $N-1$ by scaling the metric. From Theorem~\ref{thm-22}, its spectral gap satisfies $\lambda_1 \geq N$. 

The next result propose an estimate depending only on $(\lambda_1-N)$, of the $L^1$-Wasserstein distance between one direction of the measure $\mu$ and the reference  measure $\mu^+_N$.

%Recall that We denote by $\mu_N$ the probability measure on $[-\sqrt{N+1},\sqrt{N+1}]$ with density
%$$
%\mu_N(dx) = Z^-(N+1-x^2)^{N/2 - 1}
%$$
%with $Z^-$ the normalization constant to enforce unit mass. 
%{\red Pourquoi sur l'intervalle $[-\sqrt{N+1},\sqrt{N+1}]$ et pas simplement sur $[0,1]$ }
\begin{thm}[The positive dimensional case]
\label{thm-1}
Let $(M,d,\mu)$ be an RCD$(N-1, N)$ metric measure space with $N>1$ and generator $\cL$.  Let $f$ be an eigenfunction of $-\cL$ with eigenvalue $N + \varepsilon$, for $\varepsilon\in[0,1)$, satisfying $\int \Gamma(f)d\mu=N/(N+1)$. Then 
$$
W_1(\mu \circ f^{-1}, \mu_N^+) \leq C\varepsilon,
$$
where $C>0$ is an explicit constant, depending only on $N$. 
\end{thm}

The proof of this result is postponed in Section~\ref{sec-33}. We first prove the two main ingredients: $L^1$-estimates on eigenfunctions, and a criterion for comparing the distribution of an eigenfunction to a Beta distribution. 

\subsection{Estimates on the first eigenfunction} \label{sect_eigen_est}

Our key technical estimate is the following: 
\begin{lem} \label{lem_small_l1_posit}
Let assume that $\cL$ satisfies the RCD$(N-1,N)$ condition. Let $f$ be an eigenfunction of $-\cL$ with eigenvalue $N + \varepsilon$, for $\varepsilon\in[0,1)$ and satisfying $\int \Gamma(f)d\mu=N/(N+1)$.  Then, 
$$
||\Gamma(f) +(1+\varepsilon)f^2- 1||_1 \leq \varepsilon C,
$$
for some constant $C$ depending only on $N$.  The value 
$$
C =4\left( 2 +  \frac{N+1}{N}\left(\frac{2}{N+1}\log 2 + \log\left(2 + \frac{2N}{(N-1)^2}\right)\right)\right) + \frac{N-1}{N(N+1)}
$$
suffices. 
\end{lem}
The requirement that $\varepsilon < 1$ is for convenience, as it allows to simplify the writing of various bounds in the proof. 

Estimates on $h=\Gamma(f) + (1+\varepsilon)f^2$ are at the core of all of the results on rigidity and stability of sharp functional inequalities in the RCD$(N-1, N)$ setting. For rigidity,~\cite[Thm.~3.7]{Ket15} proves the $\varepsilon = 0$ case of our lemma. In the smooth unweighted setting~\cite{Ber07} uses $L^p$-estimates on $\Hess f + f$ (which is related to the gradient of $\Gamma(f) + f^2$), while~\cite{Aub05} uses $L^\infty$-estimates on $\Gamma(f) + f^2$. In the non-smooth setting,~\cite{CMS19} establishes $L^2$-estimates on $\Hess f + f$ along one-dimensional needles. 

Unlike these previous work, we use a weaker norm to estimate $h$. One of the upsides is that it is easier to work with first order derivatives (instead of Hessians) in the non-smooth setting. But the main upside is that the quantitative bounds we derive are \emph{stronger}. Indeed, one could use the self improvement in the Bakry-Emery-Bochner bound~\cite{Sav} to estimate Hessian-like quantities (this is the approach used in~\cite{BF21, Ket15} for example), but the quantitative bounds are of order $\sqrt{\varepsilon}$ instead of $\varepsilon$. Our use of an $L^1$-norm leads instead to a \emph{sharp} quantitative bound. 

Our proof of Lemma~\ref{lem_small_l1_posit} is based on the following result. 
\begin{lem} \label{small_diff_posit}
	Let assume that $\cL$ satisfies the RCD$(N-1,N)$ condition with $N>1$. Let $f$ be an eigenfunction of $-\cL$, with eigenvalue $N+\varepsilon$ for some $\varepsilon  \in [0,1)$. Then for $h = \Gamma(f) +(1+\varepsilon)f^2$ we have 
	$$
	||\cL P_sh||_1 \leq 4N\varepsilon||f^2||_1
	$$ 
where $P_s$ is the semigroup generated from $\cL$. In particular, in the smooth setting,
$$
	||\cL h||_1 \leq 4N\varepsilon||f^2||_1.
	$$ 
\end{lem}

The use of the semigroup $P_s$ in the non-smooth setting is to avoid giving a meaning to $\cL \Gamma(f)$. 
To simplify the exposition, we shall first prove Lemma~\ref{small_diff_posit} in the smooth setting, (assuming the Bochner inequality~\eqref{def_cd}) and then in the general RCD setting. 

\begin{proof}[Proof of Lemma~\ref{small_diff_posit} in the smooth setting]

Applying the smooth $CD(N-1,N)$ condition~\eqref{def_cd} to the eigenfunction $f$ (recall that $f$ is a smooth function since we are working in a smooth setting), we have
$$
\frac{1}{2}\cL \Gamma(f) + (1+\varepsilon)\Gamma(f) - \frac{(N + \varepsilon)^2}{N}f^2 \geq 0.
$$
Moreover, since $2\Gamma(f) = \cL(f^2) + 2(N+\varepsilon)f^2$, we have
$$
\cL(\Gamma(f) +(1+\varepsilon)f^2) +2(\varepsilon(N-1) + \varepsilon^2(1-1/N))f^2 \geq 0.
$$
In particular, 
$$
(\cL h)_- \leq 2\varepsilon(N-1)(1 + \frac{\varepsilon}{N})f^2\leq 2\varepsilon Nf^2,
$$
using $\varepsilon < 1$ (which is only used here to make notations less cluttered) and $N>1$. Since $\cL h$ has zero average with respect to $\mu$, we get
$$
||\cL h||_1 = 2||(\cL h)_-||_1 \leq 4\varepsilon N||f^2||_1. 
$$
Which concludes the proof. 
\end{proof}

\begin{proof}[Proof of Lemma~\ref{small_diff_posit} in the general RCD case]
	
	As mentioned previously, to follow the above scheme in the general RCD setting, we encounter the problem of giving a meaning to $\cL \Gamma(f)$. We shall exploit the particular structure of $f$ as an eigenfunction to bypass this issue via a regularization procedure. 
	
	The weak Bochner inequality~\eqref{24} applied to the eigenfunction $f$ of $-\cL$ with eigenvalue $N+\varepsilon$ takes the form:
	$$
	\frac{1}{2}\int{\cL g \Gamma(f)d\mu} + (N+\varepsilon) \int{g\Gamma(f)d\mu} \geq (N-1)\int{g\Gamma(f)d\mu} + \frac{(N+\varepsilon)^2}{N}\int{gf^2 d\mu},
	$$
	for any test function $g$ such that $g, \cL g\in L^\infty(\mu)$.
	
	If we consider test functions of the form $P_s g$ for $s > 0$ and $g, \cL g\in L^\infty(\mu)$, since $\cL$ and $P_s$ commute and $\cL P_s\Gamma(f)$ is well defined we have
	\begin{equation} \label{eq_bochner_int_reg}
	\frac{1}{2}\int{g (\cL P_s \Gamma(f))d\mu} \geq -(\varepsilon+1)\int{gP_s \Gamma(f)d\mu} + \frac{(N+\varepsilon)^2}{N}\int{g P_s(f^2)d\mu}.
	\end{equation}
	We can now remove the restriction that $\cL g$ is $L^\infty(\mu)$ by approximating a $g \in L^\infty(\mu)$ with 
	$$
	S_\varepsilon g := \int_0^\infty{P_{\varepsilon r}g \kappa(r)dr},
	$$
	where $\kappa$ is a smooth nonnegative function compactly supported in $(0, \infty)$ with $\int{\kappa(r)dr} = 1$. One can check that
	$$
	\cL S_\varepsilon g = -\frac{1}{\varepsilon}\int_0^\infty{P_{\varepsilon r}g \kappa'(r)dr}
	$$
	is indeed $L^\infty(\mu)$, and that $S_\varepsilon g$ converges to $g$ in $L^2(\mu)$ when $\varepsilon\rightarrow 0$. Therefore~\eqref{eq_bochner_int_reg} holds for $g \in L^\infty(\mu)$. But since $P_s \Gamma(f)$ is a well-defined function, we deduce the pointwise inequality
	$$
	\frac{1}{2}\cL P_s\Gamma(f) \geq -(\varepsilon+1)P_s\Gamma(f) + \frac{(\varepsilon+N)^2}{N}P_s (f^2),\quad s>0.
	$$
	Since $\Gamma(f) = \frac{1}{2}\cL f^2 + (\varepsilon+n) f^2$, we get as in the smooth setting, for $h = \Gamma(f) + (1+\varepsilon)f^2$,
	$$
	(\cL P_s h)_- \leq 2\varepsilon NP_s(f^2).
	$$
	Hence we can deduce that $||\cL P_s h||_1 \leq 4N\varepsilon||f||_2^2$. %We can then let $s$ go to zero to conclude that Lemma~\ref{lem_small_l1_posit} also holds in the non-smooth case. {\red question idiote : on peut prendre la limite quand $s$ tend vers 0 même si on ne sait pas si la quantité est dans $L^1$???} {\red Max : c'est effectivement peut etre mieux de d'abord appliquer la Proposition 3.4 pour justifier que $P_s h$ est petit dans $L^1$, et de ne faire tendre $s$ vers $0$ qu'apr\`es...}
\end{proof}

Since the kernel of $\cL$ is the set of constants,  we expect  that $h$ is concentrated around its average. A tool for proving this is provided by the following result. 

\begin{prop}
\label{prop-35}
Let assume that $\cL$ satisfies the condition RCD$(N-1,N)$ with  $N > 1$, then  for any function $g \in D(\cL)$ with $\int gd\mu=0$,  we have
$$
||g||_1 \leq \left( 2 +  \frac{N+1}{N}\left(\frac{2}{N+1}\log 2 + \log\left(2 + \frac{2N}{(N-1)^2}\right)\right)\right)||\cL g||_1.
$$
%{\red ce n'est certainement pas la constante optimale mais au moins elle est simple et explicite, en espérant qu'elle soit juste, calculs facile mais je peux me tromper surtout qu'elle est normalement fausse si $N=\infty$, à revoir}
\end{prop}

\begin{proof}
We have for any  $t>0$,  
$$
||g||_1 \leq ||g - P_t g||_1 + ||P_tg||_1.
$$
For any bounded  test function $u$, we have
$$
\int\!{(g-P_tg)u d\mu}\! =\! \int_0^t{\int\!{(-\cL P_s g)u d\mu}ds}\!=\! \int_0^t{\int\!{(-\cL g)(P_su) d\mu}ds} \leq t||u||_\infty ||\cL g||_1.
$$
Taking the supremum over all $u$ with $||u||_\infty \leq 1$, we get
$$
||g-P_tg||_1 \leq t||\cL g||_1,
$$
hence 
\begin{equation}
\label{16}
||g||_1 \leq (t+1)||\cL g||_1 + ||P_{t+1}g||_1.
\end{equation}
Since $N>1$, the operator $\cL$ satisfies the weaker condition CD$(N-1,N+1)$. From this remark, $\cL$ satisfies a Sobolev inequality (cf.~\cite{Ili83} for the smooth setting and extended to the RCD setting in~\cite{Pro15}), 
$$
||f||_{2\frac{N+1}{N-1}}^2\leq ||f||_2^2+ B\int \Gamma(f)d\mu,
$$
with $B=\frac{4N}{(N+1)(N-1)^2}$ and for functions such that terms are well defined.  And then, the ultracontractive bound 
\begin{equation}
\label{15}
||P_t g||_\infty \leq Ct^{-\frac{N+1}{2}}||g||_1, \hspace{3mm} t \leq 1
\end{equation}
with $C=\left(2+  \frac{2N}{(N-1)^2} \right)^{\frac{N+1}{2}}$, see~\cite[Thm.~6.3.1 and Rmk.~6.3.2]{BGL14}. Using this bound for $t = 1$ and the spectral gap, we have
\begin{align*}
||P_{t+1}g||_1 &\leq ||P_{t+1}g||_2 \leq e^{-Nt}||P_1g||_2 \\
&\leq e^{-Nt}\left(2+  \frac{2N}{(N-1)^2} \right)^{\frac{N+1}{2}}||g||_1.
\end{align*}

Taking $t = \frac{N+1}{2N}\left(\frac{2}{N+1}\log 2 + \log\left(2 + \frac{2N}{(N-1)^2}\right)\right)$ and using inequality~\eqref{16}, we get
$$
||g||_1 \leq (t+1)||\cL g||_1 + \frac{||g||_1}{2}.
$$
Hence
$$
||g||_1 \leq 2(t+1)||\cL g||_1 = \left( 2 +  \frac{N+1}{N}\left(\frac{2}{N+1}\log 2 + \log\left(2 + \frac{2N}{(N-1)^2}\right)\right)\right)||\cL g||_1,
$$
which is the inequality desired. 
\end{proof}

%Moreover,  since the spectral gap of $-\cL$ is bounded from below by $n$, the variance exponentially decays toward 0. Since $\int gd\mu=\int P_tgd\mu=0$,  we have for every $t\geq0$, 
%\begin{equation}
%\label{12}
%||P_{t} g||_2=\int (P_{t} g)^2d\mu-\Big(\int P_{t}gd\mu\Big)^2 \leq e^{-2n t}||P_sg||_2.
%\end{equation}
%Finally, using~\eqref{15} and~\eqref{12} we get for $t\geq s\geq 0$, 
%\begin{equation*}
%||P_tg||_1 \leq ||P_t g||_2=||P_s (P_{t-s}g)||_2 \leq e^{-{2n} (t-s)}||P_sg||_2 \leq C\frac{e^{-2N(t-s)}}{s^{(N+2)/4}}||g||_1.
%\end{equation*}
%Taking $s=1$, we can find $t(N)$ large enough (in a way that only depends on $N$) so that $C e^{-2N(t-1)}\leq 1/2$. Now, from~\eqref{16}, we get 
%$$
%||g||_1\leq 2t(N)||\cL g||_1,
%$$
%which is the inequality we were aiming for.

%If $N > 2$, the constant $C$ can be taken as $(2(N-1)/(N-2))^{N/2}$, see~\cite[Remark 6.3.2]{BGL14}. %This ultracontractive bound is equivalent to the Sobolev inequality with sharp constant,. 
%%
%%Cette dernière référence, est-elle utile.? 
%%
%Using this bound for $C$, we can take $s= 1$ and 
%$$t = \frac{\log N}{2} + \frac{1}{2} \log\left(\frac{2e^2(N-1)}{N-2}\right)$$
%to get
%$$||g||_1 \leq \left(\frac{\log N}{2} + \frac{1}{2} \log\left(\frac{2e^2(N-1)}{N-2}\right)\right)||Lg||_1.$$This concludes the proof. 
%%}
\begin{proof}[Proof of Lemma~\ref{lem_small_l1_posit}]
	Let apply Proposition~\ref{prop-35} to 
	$$
	P_s\Gamma(f)+(1+\varepsilon)P_s(f^2)-\frac{N}{N+1}\left(1+\frac{1+\varepsilon}{N+\varepsilon}\right)
	$$
	for some $s > 0$, which has zero average since $\int f^2d\mu=(N+\varepsilon)^{-1}\int{\Gamma(f)d\mu}$. We obtain, 
	$$
	\left|\left|P_s\Gamma(f)+(1+\varepsilon)P_s(f^2)-\frac{N}{N+1}\left(1+\frac{1+\varepsilon}{N+\varepsilon}\right)\right|\right|_1\leq C||\cL \big[P_s\Gamma(f)+(1+\varepsilon)P_s(f^2)\big]||_1,
	$$
where $C$ is given by Proposition~\ref{prop-35}. And then, using  Lemma~\ref{small_diff_posit},
	$$
	\left|\left|P_s\Gamma(f)+(1+\varepsilon)P_s(f^2)-\frac{N}{N+1}\left(1+\frac{1+\varepsilon}{N+\varepsilon}\right)\right|\right|_1\leq 4CN\varepsilon||f||_2 \leq 4C\varepsilon.
	$$
	 Finally, 
	$$
	\left|\left|P_s\Gamma(f)+(1+\varepsilon)P_s(f^2)-1\right|\right|_1\leq \varepsilon\Big(4C  +\frac{N-1}{N(N+1)}\Big),
	$$ 
 from the triangle inequality. We let $s$ go to zero to conclude the proof. 
\end{proof}

\subsection{Approximate Beta distribution for eigenfunctions}

The last main ingredient of the proof is a result stating that (normalized) eigenfunctions with eigenvalue close to $N$ and such that $\Gamma(f) + f^2$ is close to a constant approximately follow a symmetrized Beta distribution. The result, of independent interest, is the following: 

\begin{thm} \label{thm_beta_eigen}
Let $f$ be an eigenfunction of a diffusion operator $-\cL$ with eigenvalue $\lambda$ and invariant probability measure $\mu$, and let $\nu=\mu\circ f^{-1}$ be the pushforward of $\mu$ by $f$. Then 
$$
W_1(\nu, \operatorname{Beta}(N/2, N/2)) \leq \Big(\frac{N^2}{4}+\frac{5N}{4}+2\Big)||\Gamma(f) + f^2 -1||_{L^1(\mu)} + \frac{|N-\lambda| }{N} ||f||_{L^2(\mu)},
$$
where $\operatorname{Beta}(N/2, N/2)=\mu_N^+$ (defined in Section~\ref{sec-22}).
\end{thm}

Note that this result is only interesting if $\lambda$ is close to $N$, and if $f$ has been normalized so that $\int(\Gamma(f) + f^2 )d\mu$ is close to one. 

This result is a variant of a result of E. Meckes~\cite[Theorem 1]{Mec09}, who proved that eigenfunctions whose gradient has small variance are close to normal. As in Meckes' work, the proof will mostly be an application of Stein's method.

\subsubsection{Stein's method for Beta distributions}

%The a priori estimate on eigenfunctions of Lemma~\ref{lem_small_l1_posit} shall be used to derive an approximate integration-by-parts formula for the pushforward of the reference meausre $\mu$ by the eigenfunction. 
To prove Theorem~\ref{thm_beta_eigen}, we shall rely on a variant of Stein's method for Beta distributions. Stein's method is a set of techniques, pioneered in~\cite{Ste72, Ste86}, for bounding distances between probability measures via such integration-by-parts formulas. We refer to~\cite{Ros11, Cha14} for recent introductions and surveys of this field. 

The following variant of Stein's lemma was proven in~\cite{Dob13, GoRe13}, for non-symmetric Beta distributions on $[0,1]$ (and the symmetrized case is an immediate consequence). The values of the constants stated here are slightly worse than those stated in~\cite[Prop.~4.2]{Dob13}, and used here for ease of writing. 

\begin{thm}[\cite{Dob13, GoRe13}]
\label{thm-35}
Let $N>1$, the $W_1$ distance between a probability measure $\nu$ supported on $[-1,1]$ to a Beta$(N/2, N/2)$ distribution can be estimated by
\begin{equation}
\label{56}
W_1(\nu, \operatorname{Beta}(N/2, N/2)) \leq \frac{1}{2}\sup \left\{\int\big[{(1-x^2)g'(x) + Nxg(x) \big]d\nu}\right\}, 
\end{equation}
where the supremum is running over all smooth function $g$ on $\R$ such that 
$$
||g||_\infty \leq \frac{2}{N}\quad {\text{and}}\quad  ||g'||_\infty \leq (2 + N).
$$
\end{thm}

\subsection{Proof of Theorem~\ref{thm_beta_eigen}}\label{stein_method_beta}

We want to apply Theorem~\ref{thm-35} to the measure $ \nu = \mu \circ f^{-1}$. We first derive an approximate integration by parts formula for the measure $ \nu $.
Let $g$ be a smooth test function with $||g||_\infty \leq 2/N$ and $||g'||_\infty \leq (2+N)$. 

Using the diffusion property~\eqref{77} and the definition of $\nu$, we have
\begin{align*}
\lambda\int{x g(x) d\nu(x)} &= \lambda\int{f g\circ f d\mu} \\ &= \int{(-\cL f) g \circ f d\mu} =
  \int{\Gamma(f, g \circ f) d\mu} = \int g' \circ f \Gamma(f)d\mu. 
\end{align*}
So 
\begin{multline*}
\int\big[\lambda xg(x)-(1-x^2)g'(x)  \big]d\nu(x)=\int g' \circ f \Gamma(f)d\mu-\int(1-x^2)g'(x)d\nu\\
=\int\big[ \Gamma(f)+f^2-1\big]g' \circ f d\mu \leq (2+N)||\Gamma(f) + f^2 - 1||_{L^1(\mu)}.
\end{multline*}
where we use the definition of $ \nu = \mu \circ f^{-1}$. Moreover
$$
\left|\int xg(x)d\nu(x)\right|=\left|\int fg(f)d\mu\right| \leq ||g||_\infty ||f||_{L^2(\mu)}\leq \frac{2}{N}||f||_{L^2(\mu)}.
$$
Therefore,  we get
\begin{multline}
 \label{ipp_non_cutoff}
\left|\int\big[Nxg(x) - \left(1-x^2\right)g'(x) \big]d\nu(x)\right| \leq\\
\left|\int\big[\lambda xg(x) - \left(1-x^2\right)g'(x) \big]d\nu(x)\right| +|N-\lambda|\left|\int xg(x)d\nu(x)\right|\leq\\
(2+N)||\Gamma(f) + f^2 - 1||_{L^1(\mu)}+2\frac{|N-\lambda|}{N}||f||_{L^2(\mu)} .
\end{multline}
%When $N > 2$, the constant $C$ can be taken as 
%$$4(N+2)\left(\frac{\log N}{2} + \frac{1}{2} \log\left(\frac{2e^2(N-1)}{N-2}\right)\right).$$
To apply Theorem~\ref{thm-35}, we still need to consider a measure supported on $[-1, 1]$. To do so, we introduce $\tilde{\nu}$ the pushforward of $\mu$ by $\phi \circ f$, with $\phi$ a cutoff function, that is $\phi(x) = x$ on $[-1, 1]$, $\phi(x) = 1$ on $[1, + \infty[$, and $\phi(x) = -1$ on $]-\infty, -1]$. 
Note that
\begin{multline}
 \label{cutoff}
\int{|f-\phi(f)|d\mu} = \int{(|f|-1)\mathbbm{1}_{|f| \geq 1} d\mu} \leq \frac{1}{2}\int{(f^2-1)\mathbbm{1}_{|f| \geq 1} d\mu}  \\
= \frac{1}{2}\int{(f^2-1)_+ d\mu} \leq \frac{1}{2}\int{(\Gamma(f) + f^2-1)_+ d\mu} \leq \frac{1}{2}||\Gamma(f) + f^2 - 1||_{L^1(\mu)}. 
\end{multline}
Therefore by a coupling argument
\begin{equation} \label{w1_bnd_cutoff}
W_1(\tilde{\nu}, \nu)\leq \int xd\nu(x)-\int xd\tilde\nu(x) \leq \int{|f - \phi(f)|d\mu} \leq \frac{1}{2}||\Gamma(f) + f^2 - 1||_1.
\end{equation}
It is therefore enough to apply Theorem~\ref{thm-35} to $\tilde{\nu}$. To do so, we shall show that it satisfies the same approximate integration by parts formula than $\nu$, up to an error of order $||\Gamma(f) + f^2 - 1||_{L^1(\mu)}$. 

First we have,
\begin{align*}
&\left|\int[f g(f) - \phi(f)g(\phi(f))]d\mu \right| \leq \left|\int{(f - \phi(f))g(f)d\mu}\right|  + \left|\int{\phi(f)(g(f) - g(\phi(f)))d\mu}\right| \\
&\hspace{3mm}\leq \frac{2}{N}||f - \phi(f)||_{L^1(\mu)} + (2+N)||f - \phi(f)||_{L^1(\mu)} \leq \Big(1 + \frac{1}{N}+\frac{N}{2}\Big)||\Gamma(f) + f^2 - 1||_{L^1(\mu)},
\end{align*}
since $||\phi(f)||_\infty\leq 1$.  Secondly, 
\begin{align*}
\left|\int\big[(1-f^2)g'(f)- (1- \phi(f)^2)g'(\phi(f))\big]d\mu\right| &\leq \left|\int{(\phi(f)^2-f^2 )g'(f)d\mu}\right| \\
& \hspace{3mm} + \left|\int{(1-\phi(f)^2)(g'(f) - g'(\phi(f)))d\mu} \right| \\
&\leq (2+N)\int{(f^2 - 1)_+ d\mu} + 0 \\
&\leq \Big(1+\frac{N}{2}\Big)||\Gamma(f) + f^2 - 1||_{L^1(\mu)}.
\end{align*}
where we used again~\eqref{cutoff}. Hence from the definition of $\tilde\nu$, we have 
\begin{multline*}
\left|\int{(Nx g(x) - (1-x^2)g'(x))d\tilde{\nu}(x)} \right|=\left|\int{(N\phi(f) g(\phi(f)) - (1-\phi(f)^2)g'(\phi(f)))d\mu} \right| \leq\\
N\left|\int{[\phi(f) g(\phi(f)) - fg(f))]d\mu} \right|+\left|\int\big[(1-f^2)g'(f)- (1- \phi(f)^2)g'(\phi(f))\big]d\mu\right|\\
+\left|\int\big[Nxg(x) - \left(1-x^2\right)g'(x) \big]d\nu(x)\right|. 
\end{multline*}
We can deduce from the previous estimates  and~\eqref{ipp_non_cutoff},
$$
\left|\int{Nx g(x) - (1-x^2)g'(x)d\tilde{\nu}} \right| \leq \Big(\frac{N^2}{2}+\frac{5N}{2}+3\Big)||\Gamma(f) + f^2 -1||_{L^1(\mu)} + 2\frac{|N-\lambda| }{N} ||f||_{L^2(\mu)}.
$$

At the end, from Theorem~\ref{thm-35} to $\tilde{\nu}$ and the estimate~\eqref{w1_bnd_cutoff}, 
\begin{multline*}
W_1(\nu, \operatorname{Beta}(N/2, N/2))\leq  W_1(\tilde{\nu}, \nu)+W_1(\tilde{\nu}, \operatorname{Beta}(N/2, N/2))\leq \\
\Big(\frac{N^2}{4}+\frac{5N}{4}+2\Big)||\Gamma(f) + f^2 -1||_{L^1(\mu)} + \frac{|N-\lambda| }{N} ||f||_{L^2(\mu)},
\end{multline*}
which is the inequality desired. 
\subsection{Proof of Theorem~\ref{thm-1}}
\label{sec-33}

We can now straightforwardly combine Theorem~\ref{thm_beta_eigen} and Lemma~\ref{lem_small_l1_posit} to conclude. 

Note that under our normalization $\int \Gamma(f)d\mu= N/(N+1)$, we have
$$
||f^2||_{L^1(\mu)}\leq \frac{1}{N+1}
$$
%$$
%\Big|\hspace{1mm}||f||_2^2-\frac{1}{N+1}\Big| = \frac{1}{N+1}\left(1-\frac{N}{N+\varepsilon}\right) \leq \frac{\varepsilon}{N^2}
%$$
and hence
$$
||\Gamma(f) + f^2 - 1||_{L^1(\mu)} \leq \varepsilon||f^2||_{L^1(\mu)} + ||\Gamma(f) + (1+\varepsilon)f^2 - 1||_{L^1(\mu)} \leq C\varepsilon.
$$
for some explicit constant $C>0$ depending only on $N$. 

\section{Stability of the spectral gap for RCD$(1,\infty)$ spaces}
\label{sec-4}
In this section, we are working with $(M,d,\mu)$ a RCD$(1, \infty)$ metric measure space with generator $\cL$ and unit mass.

Arguing as in the finite-dimensional case, we can use the Bochner formula to get the following estimate on the gradient of a normalized eigenfunction $f$: 
\begin{equation}
||\cL P_s\Gamma(f) ||_1 \leq C\varepsilon, 
\end{equation}
for all $s > 0$. 

However, unlike RCD$(N-1,N)$ spaces, in this situation we do not have ultracontractive estimates on the semigroup (inequality~\eqref{15})  to prove an $L^1$-inequality as in Lemma~\ref{lem_small_l1_posit}. Indeed, as we shall see in Proposition~\ref{counterexemple_gauss_l1} below, that inequality fails for the Gauss space. Therefore, we must rely on a weaker functional inequality: 
\begin{lem} \label{lem_l1_diff_est}
Let assume that the spectral gap of $-\cL$ is greater than $1$. Let $g$ such that $\int gd\mu=0$ and $g\in L^p(\mu)$ for some $p > 1$.  Then
$$
\left|\left|g\right|\right|_1 \leq C||\cL g||_1\left(1 + \log\left(\max\Big(\frac{||g||_p }{||\cL g||_1},1\Big)\right)\right),
$$
for some constant $C$ depending only on $p$.  %{\red j'ai du changer pour un problème de log avec un Max, de la fonction maximum!  }
\end{lem}

Note that this lemma does not involve any assumption on curvature. We will actually be interested in applying it when $||g||_p$ is bounded and $||\cL g||_1$ is small. 

\begin{proof}
%We assume without loss of generality that $\int{gd\mu} = 0$. 
As before, using inequality~\eqref{16},  we obtain for $t\geq0$, 
$$
||g||_1 \leq t||\cL g||_1 + ||P_tg||_1\leq  t||\cL g||_1 + ||P_tg||_p.
$$
According to~\cite[Thm~1.6]{CGR10}, we have for some constant $C_p$ (depending only $p$), 
$$
||P_t g||_p \leq C_p \exp(-4(p-1) t/p^2)||g||_p.
$$
Hence, we have
$$
||g||_1 \leq t||\cL g||_1 +  C_p \exp(-4(p-1) t/p^2)||g||_p.
$$
Assume for simplicity that $C_p\geq1$ and take 
$$
t = \frac{p^2}{4(p-1)}\log\Big[\max\Big(\frac{||g||_p }{||\cL g||_1};1\Big)C_p\Big],
$$
in order to get 
$$
 C_p \exp(-4(p-1) t/p^2)\leq \frac{||\cL g||_1}{||g||_p }.
$$
 we get
$$
||g||_1 \leq ||\cL g||_1\left[1 + \frac{p^2 }{4(p-1)}\log\left(C_p\max\Big(\frac{||g||_p }{||\cL g||_1},1\Big)\right)\right], 
$$
which is the inequality expected. 
\end{proof}
\begin{lem}
	\label{lem-42}
	Let assume that $\cL$ satisfies the RCD$(1,\infty)$ condition. And let $f$ be an eigenfunction of $-\cL$ with eigenvalue $1 + \varepsilon$, for $\varepsilon\in[0,1]$ and satisfying $\int \Gamma(f)d\mu=1$.  Then, 
	\begin{equation}
	\label{120}
	||P_s\Gamma(f) +\varepsilon P_s(f^2)- 1-\varepsilon||_1 \leq C\varepsilon\log(2/\varepsilon),
	\end{equation}
	for some numerical  $C$. 
\end{lem}
\begin{proof}
Following the same proof as Lemma~\ref{small_diff_posit} in the general RCD setting, we get 
$$
||\cL (P_s\Gamma(f)+\varepsilon P_s(f^2))||_1\leq 4\varepsilon||f^2||_1\leq 4\varepsilon,
$$	
since $\int f^2d\mu=(1+\varepsilon)^{-1}\leq1$. 

If $h=P_s(\Gamma(f)+\varepsilon f^2)$, we have $\int hd\mu=(1+2\varepsilon)(1+\varepsilon)^{-1}$ and then 
$$
\Big|\Big|h- \frac{1+2\varepsilon}{1+\varepsilon}\Big|\Big|_1 \leq C||\cL h||_1\left(1 + \log\left(\max\Big(\frac{||h||_p }{||\cL h||_1},1\Big)\right)\right)\leq 4C\varepsilon \left(1 + \log\left(\max\Big(\frac{||h||_p }{4\varepsilon},1\Big)\right)\right). 
$$
Hence
$$
||h-1||_1\leq \Big|\Big|h- \frac{1+2\varepsilon}{1+\varepsilon}\Big|\Big|_1+\frac{1+2\varepsilon}{1+\varepsilon}-1 \leq4C\varepsilon \left(1 + \log\left(\max\Big(\frac{||h||_p }{4\varepsilon},1\Big)\right)\right) + \varepsilon.
$$
From~\cite[Prop.~3.2]{BF21} (and the RCD$(1,\infty)$ condition) we have, 
\begin{multline*}
||h||_p\leq ||P_s\Gamma(f)||_p+||P_sf||_{2p} \leq ||\Gamma(f)||_p+||f||_{2p}\\
\leq (8p-4)^{(1+\varepsilon)/2}||f^2||_1+(p-1)^{(1+\varepsilon)/2}||f^2||_1\leq D
\end{multline*}
for some other constant $D>0$, which is the last estimate to prove~\eqref{120}.
\end{proof}

Lemma~\ref{lem-42} improves on a particular case of~\cite[Lem.~3.3]{BF21}, and is enough to improve~\cite[Thm.~1.3]{BF21} into:
\begin{thm}
\label{thm-42}
Let $(M,d,\mu)$ a RCD$(1, \infty)$ metric measure space with generator $\cL$.  Let  $f$ is an  eigenfunction of the $-\cL$ with eigenvalue  $1 + \varepsilon$ ($\varepsilon\in[0,1]$) and satisfying  $\int \Gamma(f)d\mu=1$, then 
$$
W_1(\mu \circ f^{-1}, \gamma) \leq C\varepsilon \log(2/\varepsilon),
$$
for some numerical constant $C>0$ and  $\gamma$ is the standard Gaussian measure. 
\end{thm}

\begin{proof}
We follow the same line of arguments as in~\cite{BF21} (and the implementation of Stein's method is essentially the same as the argument in~\cite{Mec09}). From Stein's lemma for the one-dimensional Gaussian distribution (see for example~\cite[Lem.~3.5]{BF21}) we have, 
$$
W_1(\mu \circ f^{-1}, \gamma) \leq\sup\left\{\int (g'(x)-g(x)x)d\mu \circ f^{-1}(x),\,\,||g'||_\infty\leq4 \right\}.
$$
Similarly we have  $(1+\varepsilon)\int{x g(x) d\mu \circ f^{-1}(x)} = \int g' \circ f \Gamma(f)d\mu$,
then 
\begin{multline*}
\int (g'(x)-g(x)x)d\mu \circ f^{-1}(x)= \frac{1}{1+\varepsilon} \int\big[1+\varepsilon - \Gamma(f)-\varepsilon f^2\big] g' \circ f d\mu\\
-\frac{\varepsilon}{1+\varepsilon}\int f^2 g' \circ f d\mu.
\end{multline*}
That is for $g$ satisfying $||g'||_\infty\leq 4$, 
$$
\int (g'(x)-g(x)x)d\mu \circ f^{-1}(x)\leq  \frac{4}{1+\varepsilon} || \Gamma(f)+\varepsilon f^2-1-\varepsilon ||_1+4\varepsilon. 
$$
From the estimate~\eqref{120} we get 
$$
\int (g'(x)-g(x)x)d\mu \circ f^{-1}(x)\leq C\varepsilon \log(2/\varepsilon),
$$
for some numerical constant $C$, which concludes the proof. 
\end{proof}
We now show that indeed the $L^1$-estimate of Lemma~\ref{lem_small_l1_posit} may fail in the RCD$(1, \infty)$ setting. 

\begin{prop} 
\label{counterexemple_gauss_l1}
The inequality 
$$
||f||_1 \leq C||\cL f||_1
$$
for some constant $C>0$ and every centered functions $f$ fails in the Gauss space $(\R, |\cdot|, \gamma)$. 
\end{prop}

The counterexample in the proof below is inspired by Naor and Schechtman's proof~\cite{NS02} that an analogous inequality on the hypercube $\{0, 1\}^d$ cannot hold with a constant that is uniform in $d$. We thank Alexandros Eskenazis for pointing out this reference to us. 

\begin{proof}
Let $\cL$ be the one-dimensional Ornstein-Uhlenbeck operator $\cL f = f'' - xf'$. We shall exhibit a family of functions $f_r$ such that the ratio $$||f_r||_1/||\cL f_r||_1$$ is unbounded.  Let $f_r$ be the centered solution (in $L^2(\gamma)$) to the Poisson equation
$$
\cL f_r = -\mathbbm{1}_{(-\infty, -r]} + \mathbbm{1}_{[r, +\infty)}.
$$
Since the source term is antisymmetric, we are looking for an antisymmetric solution (which will then be centered). We can check that $f'_r$ is given on $\R_+$ by the formula
$$
f_r'(x) = \begin{cases}
-\sqrt{2\pi}(1-\varphi(r))e^{x^2/2} & {\text{ if } } 0\leq x \leq r; \\
-\sqrt{2\pi}(1-\varphi(x))e^{x^2/2} & \text{ if } x > r,
\end{cases}
$$
and extended by symmetry to $\R$. Here $\varphi$ stands for the Gaussian cumulative distribution function $\varphi(x) = \int_{-\infty}^x {(2\pi)^{-1/2}\exp(-t^2/2)dt}$.  Then for $x \in [0, r]$ we have
$$
f_r(x) = -\sqrt{2\pi}(1-\varphi(r))\int_0^x{e^{t^2/2}dt}.
$$
It is easy to check that for $x \in [1, r]$ we have a lower bound of the form
$$
f_r(x) \geq C(1-\varphi(r))\frac{e^{x^2/2}}{x},\,\,C>0.
$$
Therefore, for $r>1$, 
\begin{equation*}
\int{|f_r|d\gamma} \geq C(1-\varphi(r))\int_{\sqrt{r}}^r {\frac{e^{x^2/2}}{x}d\gamma} = \frac{C}{\sqrt{2\pi}}(1-\varphi(r))\int_{\sqrt{r}}^r{ x^{-1}dx} = \frac{C}{2\sqrt{2\pi}}(1-\varphi(r))\log r.
\end{equation*}
On the other hand
$$
||\cL f_r||_1 = \int{|\cL f_r|d\gamma} = 2(1-\varphi(r)).
$$
Hence $||f_r||_1/||\cL f_r||_1$ is unbounded as $r$ goes to infinity, which concludes the proof. 
\end{proof}

\section{The negative dimension case}
\label{sec-5}

%{\color{red} Verifier les questions de r\'egularit\'e, et l'existence de la fonction propre. Au pire on pourra ajouter \c ca comme hypoth\`eses. et bien entendu l'intégrabilité des fonctions propres. } 

We now consider a space satisfying the curvature-dimension condition condition with \emph{negative dimension parameter} $N$. We refer to~\cite{Sch03, Oht16, KM18} for an introduction
 to the notion. Since the Bochner inequality in the non-smooth setting seems not to have been investigated yet in the literature (in particular, its possible equivalence with other non-
 smooth definitions studied in~\cite{Oht16}), we shall restrict ourselves to the smooth setting, and assume the Bochner inequality~\eqref{def_cd} holds in a strong sense, as 
 in~\cite{KM18}.  As noted in~\cite{KM17}, the spectral gap bound of Theorem~\ref{thm-22} is sharp for $N \leq -1$, but ceases to be sharp when $N$ is negative and $|N|$ is small (see also~\cite{Mai19}  where RCD spaces with negative effective dimension, positive curvature and \emph{infinite} volume are studied). 
 When $N\leq -1$, a particular example of a model satisfying a CD$(1-N,N)$ condition  is given by the generalized Cauchy distribution as presented in Section~\ref{sec-22},
$$
d\mu_N^- = \frac{(1+x^2)^{N/2-1}}{Z^-}dx
$$ 
on $\R$, with generator 
$$
\cL^- f = (1+x^2)f'' +N x f'.
$$

%Unlike in the positive dimensional setting, there are RCD spaces with negative effective dimension, positive curvature and \emph{infinite} volume~\cite{Mai19}. We shall only 
%consider the situation where the total volume is finite (and normalized to take value $1$). {\red pas certain de comprendre, on dit que RCD n'est pas clair et que l'on fait dans le cas 
%lisse....}

Our main result in this section is: 
\begin{thm}[The negative dimensional case]
Let $(M, g )$ be a smooth Riemannian manifold associated with the probability measure $\mu$,  satisfying the curvature-dimension condition  CD$(1-N,N)$ with $N < -1$. We 
assume that the spectral gap $\lambda_1$ satisfies $\lambda_1 \leq -N + \varepsilon$ for some $\varepsilon \in (0,1)$, and let $f$ an eigenfunction of $\cL$ with eigenvalue 
$\lambda_1$ such that  
 $\int \Gamma(f)d\mu= N/(N+1)$. 
 
 Assume moreover that $\Gamma(f) \in L^{1+c}$ for some $c > 0$. There is a constant $C> 0$ depending only on $N$ and $c$  such that 
$$
W_1(\mu \circ f^{-1}, \mu_N^-) \leq C\varepsilon \log(2/\varepsilon).
$$
\end{thm}

This result shall be a direct consequence of combining Proposition~\ref{prop-53} with Corollary~\ref{cor_eigen_cauchy} below. 

\subsection{Estimates on the first eigenfunction}
\label{sec-51}

As for Lemma~\ref{small_diff_posit} with the same proof, we deduce the following estimate for the Bochner formula: 

\begin{lem}
	\label{lem-51}
Let assume that $\cL$ satisfies the CD$(1-N,N)$ condition. We assume that there exits  $f$ an eigenfunction of $-\cL$  with eigenvalue $-N + \varepsilon$ with $\varepsilon\in[0,1]$ such that $f\in L^2(\mu)$. Then 
\begin{equation}
\label{49}
||\cL (\Gamma(f) +(\varepsilon-1)f^2)||_1 \leq 4\varepsilon\frac{(1-N)^2}{|N|}||f^2||_1.
\end{equation}
\end{lem}

\begin{proof}
Applying the CD$(1-N,N)$ condition~\eqref{def_cd} to the eigenfunction $f$, we have
$$
\frac{1}{2}\cL \Gamma(f) + (\varepsilon-1)\Gamma(f) - \frac{(N - \varepsilon)^2}{N}f^2 \geq 0.
$$
Moreover, since $2\Gamma(f) = \cL(f^2) + 2(\varepsilon-N)f^2$, we have
$$
\frac{1}{2}\cL(\Gamma(f) +(\varepsilon-1)f^2) +\varepsilon(\varepsilon-N)\frac{1-N}{|N|}f^2 \geq 0
$$
In particular, for $h=\Gamma(f) +(\varepsilon-1)f^2$, 
$$
(\cL h)_- \leq 2\varepsilon(\varepsilon-N)\frac{1-N}{|N|}f^2\leq -2\frac{\varepsilon}{N}\left( 1-N\right)^2f^2,
$$
using $\varepsilon < 1$. Since $\cL h$ has zero average with respect to $\mu$, we get
$$
||\cL h||_1 = 2||(\cL h)_-||_1 \leq 4\frac{\varepsilon}{|N|}\left( 1-N\right)^2||f^2||_1,
$$
which concludes the proof. 
\end{proof}

Once again, we do not have an ultracontractive estimate for the semigroup when the dimension parameter is negative, so we must rely on Lemma $\ref{lem_l1_diff_est}$ to deduce 
an $L^1$-estimate on $h=\Gamma(f) +(\varepsilon-1)f^2$. This use of Lemma $\ref{lem_l1_diff_est}$ would be justified if both $\Gamma(f)\in L^p(\mu)$ for some $p>1$ and $f\in L^q(\mu)$ for 
some $q>2$. The following lemma shows that only the integrability condition on $\Gamma(f)$ is actually required.

\begin{lem}
\label{lem-52}
Assume the couple $(\mu, \Gamma)$ satisfies a Poincar\'e inequality with constant $C_P$.  If $\Gamma(g) \in L^{1+ c}(\mu)$ for some $c > 0$ then $g \in L^{2(1+2c)/(1+c)}(\mu)$, and 
$$
||g||_{2(1+2c)/(1+c)}^{2(1+2c)/(1+c)} \leq ||g||_2^{2(1+2c)/(1+c)} + 4C_P||\Gamma(g)||_{1+c}||g||_2^{2c/(1+c)}. 
$$
 \end{lem}

\begin{proof}
The Poincaré inequality insures that for all smooth function $h$, 
$$
\int h^2d\mu-\left(\int hd\mu\right)^2\leq C_P\int\Gamma(h)d\mu.
$$
This inequality applied to  $g^p$ gives
$$
||g||^{p}_{2p} - ||g||^{2p}_p \leq p^2C_P \int{\Gamma(g)g^{2(p-1)}d\mu}.
$$
We then apply the H\"older inequality with exponents $1+c$ and $(c+1)/c$ to get
$$
||g||_{2p}^{2p} - ||g||_p^{2p} \leq p^2C_P||\Gamma(g)||_{1+c}\left(\int{g^{2(p-1)(c+1)/c}d\mu}\right)^{c/(c+1)}.
$$
We then take $p = 1 + c/(c+1)$, since $p < 2$ we have $||g||_p^{2p}\leq ||g||_2^{2p}$, which concludes the proof.
\end{proof}

\begin{prop}
	\label{prop-53}
Assume that $\cL$ satisfies the  CD$(1-N, N)$  condition (with $N < -1$) and assume that $\cL $ admits an eigenfunction $f$ with eigenvalue $-N + \varepsilon$ with $\varepsilon \leq 1$. We assume moreover that  $f\in L^2(\mu)$,  $\Gamma(f) \in L^{1 + c}(\mu)$ for some $c > 0$, and $\int \Gamma(f)\,d\mu=\frac{N}{N+1} $. Then
$$||\Gamma(f) -f^2 - 1||_1 \leq C\varepsilon\log(2/\varepsilon)$$
for a numerical constant $C$ that only depends on $c$ and $N$. 
\end{prop}

\begin{proof}
The proof is a straightforward combination of Lemmas~\ref{lem_l1_diff_est},~\ref{lem-51} and~\ref{lem-52}. Note that with the normalization of $\int\Gamma(f)d\mu=\frac{N}{N+1}$, we have 
$$
\int f^2d\mu = \frac{N}{(N+1)(\varepsilon - N)}
$$
 and 
$$
\int{(\Gamma(f) - f^2) d\mu} = 1 + \frac{\varepsilon}{(N+1)(N-\varepsilon)}.
$$ 
\end{proof}

\subsection{Stein's method for generalized Cauchy distributions}
In this Section, we shall establish a version of Stein's lemma for generalized Cauchy distributions in dimension one, following the standard approach of explicitly solving a Poisson equation. The result is the following:

\begin{thm}[Stein's method for Cauchy distributions]
\label{stein_lem_cauchy_gen}
Let $N<-1$, the $W_1$ distance between a probability measure $\nu$ on $\R$ to the generalized Cauchy distribution $\mu_N^-$ satisfies
\begin{equation}
W_1(\nu, \mu_N^-) \leq \sup \left\{\int\big[{(1+x^2)g'(x) + Nxg(x) \big]d\nu}\right\}, 
\end{equation}
where the supremum is running over all absolutely continuous function $g$ on $\R$ such that 
\begin{align*}
||g||_\infty & \leq \max\left(\frac{4N+3}{|N|(N+1)}\,,\,\frac{9}{2}\left(\frac{N}{N+1}\right)^2+\frac{N}{N+1} \right),\\
\quad {\text{and}}\quad & ||g'||_\infty  \leq 1 + \left( \frac{3}{2}+\frac{N}{N+1}\right)\frac{N}{N+1}.
\end{align*}
\end{thm} 

The constants appearing here are not sharp, but have the sharp order of magnitude as $N$ goes to $-\infty$. 

As a corollary, we have the analogue of~\cite[Thm.~1]{Mec09} for generalized Cauchy target distributions: 

\begin{cor} \label{cor_eigen_cauchy}
Let $f$ be an eigenfunction of a diffusion operator $-\cL$ with eigenvalue $\lambda$ and invariant probability measure $\mu$, and let $\nu$ be the pushforward of $\mu$ by $f$. Then for $N < -1$ we have
\begin{multline*}
W_1(\nu, \mu_N^-) \leq \left(1 + \left( \frac{3}{2}+\frac{N}{N+1}\right)\frac{N}{N+1}\right)||\Gamma(f) -f^2 - 1||_{L^1(\mu)}\\
+ |\lambda + N|\max\left(\frac{4|N|+3}{|N||N+1|}\,,\,\frac{9}{2}\left(\frac{N}{N+1}\right)^2+\frac{N}{N+1} \right)||f^2||_{L^1(\mu)}.
\end{multline*}
\end{cor}

\begin{proof}[Proof of Corollary~\ref{cor_eigen_cauchy}]
Once again, let $ \nu = \mu \circ f^{-1}$ and $g$ be a smooth test function. 
\begin{align*}
\lambda\int{x g(x) d\nu} &= \lambda\int{f g\circ f d\mu} = \int{(-\cL f) g \circ f d\mu} =  \int{\Gamma(f, g \circ f) d\mu} = \int{g' \circ f\, \Gamma(f)d\mu} \\
&= \int g'\circ f \, \left(\Gamma(f)-f^2-1 \right)\,d\mu + \int (x^2+1)g'd\nu.
\end{align*}
Hence $$\left|\int\big[(1+x^2)g'(x) + Nxg(x) \big]d\nu\right| \leq ||g'||_\infty||\Gamma(f) -f^2 - 1||_1 + |\lambda + N|||g||_\infty||f||_2.$$
We can then apply Theorem~\ref{stein_lem_cauchy_gen} to conclude. 
\end{proof}

The proof of Theorem~\ref{stein_lem_cauchy_gen} relies on Stein's method, as briefly presented in Section~\ref{stein_method_beta}. To begin with, we notice that the probability measure $\mu_N^-:=\frac{1}{Z^-}(1+x^2)^{\frac{N}{2}-1}$ on $\R$, with negative dimension $N<-1$, is characterized by the integration-by-parts formula
$$\int{(1 + x^2)g'(x)d\mu_N(x)} = -\int{Nxg(x)d\mu_N(x)}.$$ 
The associated Stein equation is then 
\begin{equation}\label{steineq}
(1+x^2)g'+Nxg = h - \int{h\,d\mu_N^-}
\end{equation}
for some function $h$.
The solution of interest is given by
\begin{align}\label{steinsolution}
g(x)=& \,(1+x^2)^{-\frac{N}{2}}\int_{-\infty}^x (1+t^2)^{\frac{N}{2}-1}\, \left(h(t) - \int{h\,d\mu_N^-}\right)dt \\
=&-(1+x^2)^{-\frac{N}{2}}\int_{x}^{+\infty} (1+t^2)^{\frac{N}{2}-1}\, \left(h(t) - \int{h\,d\mu_N^-}\right)\,dt
\end{align}

\begin{proof}[Proof of Theorem $\ref{stein_lem_cauchy_gen}$]
For all absolutely continuous $h$ such that $||h||_{Lip} \leq 1$, let $g_h$ be the solution ($\ref{steinsolution}$) of $$(1+x^2)g_h'+Nxg_h = h - \int{h\,d\mu_N^-}. $$ Then we have
$$ W_1(\nu, \mu_N^-) = \sup_{||h||_{Lip} \leq 1} \int{\left(h - \int{h\,d\mu_N^-}\right)d\nu}  
= \sup_{g_h} \int{\left((1+x^2)g'_h+Nxg_h\right)d\nu},
$$ and bounds on $g_h$ and $g_h'$ given by Lemma $\ref{steinnegatif}$ below complete the proof.
\end{proof}

To prove Lemma $\ref{steinnegatif}$, we need the technical lemma below.

\begin{lem}\label{tailcontrol}
We have the following estimates on the cumulative distribution function of $\mu_N$:
$$\forall x <0,\quad \int_{-\infty}^x \frac{1}{Z^-}(1+x^2)^{\frac{N}{2}-1}\,dt \leq \min \left(\frac{1}{2}, \frac{1}{|NZ^- x|} \right)(1+x^2)^\frac{N}{2},$$
$$\forall x > 0,\quad \int_{x}^{+\infty} \frac{1}{Z^-}(1+t^2)^{\frac{N}{2}-1}\,dt \leq \min \left(\frac{1}{2}, \frac{1}{|NZ^- x|} \right)(1+x^2)^\frac{N}{2}. $$
Moreover,
 $$ \forall x\leq 0,\quad \int_{-\infty}^x \frac{1}{C_N}(1+t^2)^\frac{N}{2}\,dt\leq \frac{1}{2}(1+x^2)^{\frac{N}{2}+1},$$ where $C_N:=\int_\R (1+t^2)^\frac{N}{2}\,dt$, and $$\forall x<0, \quad\int_{-\infty}^x \frac{1}{C_N}(1+t^2)^\frac{N}{2}\,dt\leq \frac{1}{(N+1)C_N\,x}(1+x^2)^{\frac{N}{2}+1}. $$
\end{lem}
\begin{proof}
These bounds can be straightforwardly established by studying the functions and their monotonicity.

Let us show the first estimate. Let $h(x):=\int_{-\infty}^x \frac{1}{Z^-}(1+x^2)^{\frac{N}{2}-1}\,dt - \frac{1}{2}(1+x^2)^\frac{N}{2}$.\\
We have $$h'(x)=(1+x^2)^{\frac{N}{2}-1}\left[\frac{1}{Z^-}-\frac{N}{2} x\right].$$ 

Hence $h'$ is decreasing on $(-\infty,\frac{2}{NZ^-})$ and increasing on $(\frac{2}{NZ^-},0)$. Now, $h(-\infty)=0$ and $h(0)=0$ because $\mu_N^-$ is symetric. Therefore, $h(x)\leq 0$.

Let $h_1(x):= \int_{-\infty}^x \frac{1}{Z^-}(1+t^2)^{\frac{N}{2}-1}\,dt - \frac{1}{NZ^- x}(1+x^2)^\frac{N}{2}$.

We have $$h_1'(x)=\frac{1}{NZ^- x^2}(1+x^2)^{\frac{N}{2}} < 0.$$ Hence $h_1(x)\leq h_1(-\infty)=0$.

The second estimate follows from the first one by symmetry.

With regards to the third estimate, let $h(x):=\int_{-\infty}^x \frac{1}{C_N}(1+t^2)^\frac{N}{2}\,dt-\frac{1}{2}(1+x^2)^{\frac{N}{2}+1}$. Then 
$$h'(x)=(1+x^2)^\frac{N}{2}\left(\frac{1}{C_N}-(\frac{N}{2}+1)x\right),$$ 
and there are two cases to consider. Either $N\leq -2$ and then $h$ is decreasing on $(-\infty,\frac{N/2+1}{C_N})$, increasing on $(\frac{N/2+1}{C_N},0)$ and $h(-\infty)=h(0)=0$, or $N\in(-2,-1)$ and then $h$ is increasing on $\R_-$ and $h(0)=0$. The result therefore stands in both cases.

Finally, let $h_1(x):= \int_{-\infty}^x \frac{1}{C_N}(1+t^2)^\frac{N}{2}\,dt-\frac{1}{(N+1)C_N\,x}(1+x^2)^{\frac{N}{2}+1}$. Then $$h_1'(x)=\frac{1}{(N+1)C_N\,x^2}(1+x^2)^{\frac{N}{2}}<0,$$ hence $h_1$ is decreasing, and $h_1(-\infty)=0$. The final estimate immediately follows.
\end{proof}

We shall now establish a priori bounds on solutions to the Poisson equation ($\ref{steineq}$).

\begin{lem}\label{steinnegatif}
If $h$ is absolutely continuous and $\mu_N^-$-centered, then the solution $(\ref{steinsolution})$ satisfies $$||g||_\infty \leq L_N ||h'||_\infty\quad \mathrm{and}\quad  ||g'||_\infty\leq K_N ||h'||_\infty $$ where 
\begin{align*}
L_N & =\max\left(\frac{4|N|+3}{N(N+1)}\,,\,\frac{9}{2}\left(\frac{N}{N+1}\right)^2+\frac{N}{N+1} \right),\\
\mathrm{and} & \quad  K_N  = 1 + \left( \frac{3}{2}+\frac{N}{N+1}\right)\frac{N}{N+1}.
\end{align*}
\end{lem}
\begin{proof}
Let $q(x):= \int_{-\infty}^x \frac{1}{Z^-}(1+x^2)^{\frac{N}{2}-1}\,dt$, and let us rewrite the solution $g$ and its derivative.
We have 
$$h(x)-\int{h\,d\mu_N^-} = \int_{-\infty}^x h'(t)q(t)\,dt - \int_x^{+\infty}h'(t)(1-q(t))\,dt$$
Combined with \eqref{steinsolution}, we get
\begin{equation}\label{solsteinreecrite}
 g(x)=-Z^- \frac{1-q(x)}{(1+x^2)^{\frac{N}{2}}}\int_{-\infty}^x h'(t)q(t)\,dt \, -\, Z^- \frac{q(x)}{(1+x^2)^{\frac{N}{2}}}\int_{x}^{+\infty}  h'(t)(1-q(t))\,dt
\end{equation}
Finally, since $g$ is solution to \eqref{steineq}, we obtain
\begin{align*}
g'(x) &= \frac{1}{1+x^2}\left(1+NZ^- (1-q(x))x(1+x^2)^{-\frac{N}{2}}\right)\int_{-\infty}^x  h'(t) q(t) \,dt \\
 &- \frac{1}{1+x^2}\left(1-NZ^- q(x)x(1+x^2)^{-\frac{N}{2}}\right) \int_{x}^{+\infty}  h'(t)(1-q(t))\,dt.
\end{align*}
Hence 
$$||g||_\infty \leq Z^-\,\underset{x\in \R}{\sup}\left( b_1(x) + b_2(x) \right)||h'||_\infty\quad \mathrm{and}\quad||g'||_\infty \leq \underset{x\in \R}{\sup}\left( a_1(x) + a_2(x) \right)||h'||_\infty $$
where
$$a_1(x):=\frac{1}{1+x^2}\left|1+NZ^- (1-q(x))x(1+x^2)^{-\frac{N}{2}}\right|\int_{-\infty}^x  q(t) \,dt,$$
$$a_2(x):=\frac{1}{1+x^2}\left|1-NZ^- q(x)x(1+x^2)^{-\frac{N}{2}}\right| \int_{x}^{+\infty}  (1-q(t))\,dt, $$
$$b_1(x):=(1-q(x))(1+x^2)^{-\frac{N}{2}}\int_{-\infty}^x q(t)\,dt, $$ and $$b_2(x):=q(x)(1+x^2)^{-\frac{N}{2}}\int_{-\infty}^x (1-q(t))\,dt. $$
It remains to show that $K_N:= \underset{x\in \R}{\sup}\left( a_1(x) + a_2(x) \right)$ and $L_N:= Z^-\,\underset{x\in \R}{\sup}\left( b_1(x) + b_2(x) \right)$ are finite. Let us begin with $K_N$. Since $\mu_N$ is symmetric, $a_1(-x)=a_2(x)$, so $a_1+a_2$ is symmetric. Moreover, with Lemma $\ref{tailcontrol}$ and since $Z^-\leq C_N$:
\begin{align*}
\forall x\leq 0,\quad a_1(x) &= (1+x^2)^{-1}\int_{-\infty}^x  q(t) \,dt+NZ^- (1-q(x))x(1+x^2)^{-\frac{N}{2}-1}\int_{-\infty}^x  q(t) \,dt\\
&\leq (1+x^2)^{-1}\int_{-\infty}^x \frac{1}{2}(1+t^2)^\frac{N}{2}\,dt + NZ^- x (1+x^2)^{-\frac{N}{2}-1}\int_{-\infty}^x \frac{1}{2}(1+t^2)^\frac{N}{2}\,dt \\
&\leq \frac{C_N}{4}(1+x^2)^{-1}(1+x^2)^{\frac{N}{2}+1} +\frac{1}{2} \frac{NZ^-}{N+1}\\
&\leq \frac{C_N}{4}+\frac{1}{2} \frac{NZ^-}{N+1} \leq \frac{1}{2} C_N\left(\frac{1}{2}+\frac{N}{N+1}\right),
\end{align*}
and
\begin{align*}\forall x>0, \quad \int_{-\infty}^xq(t)\,dt &= \int_{-\infty}^0 q(t)\,dt + \int_{0}^x q(t)\,dt\\
 &\leq \frac{1}{2}\int_{-\infty}^0 (1+t^2)^\frac{N}{2}dt + \int_0^x \,dt \\
 &= \frac{C_N}{4} + x.
\end{align*}
Hence,
\begin{align*}
\forall x>0,\quad a_1(x) &= \left|(1+x^2)^{-1}\int_{-\infty}^x  q(t) \,dt+NZ^- (1-q(x))x(1+x^2)^{-\frac{N}{2}-1}\int_{-\infty}^x  q(t) \,dt\right|\\
&\leq (1+x^2)^{-1}\left(\frac{C_N}{4} + x\right) + (1+x^2)^{-1}\left(\frac{C_N}{4} + x\right)\\
&= 1+\frac{C_N}{2}
\end{align*}
Finally, $$K_N= \underset{x\in\R}{\sup}\left( a_1(-x)+a_1(x)\right) \leq 1+\frac{1}{2}\left( \frac{3}{2}+\frac{N}{N+1}\right)\int_\R(1+t^2)^\frac{N}{2}\,dt \leq 1 + \left( \frac{3}{2}+\frac{N}{N+1}\right)\frac{N}{N+1} .$$
To bound $L_N$ the work is very similar. First, we notice the symmetry $b_1(-x)=b_2(x)$, then using many times again Lemma $\ref{tailcontrol}$, we get
\begin{align*}
\forall & x\leq -1, \quad b_1(x)\leq\frac{1}{|N|Z^-},\\
\forall & x \in [-1,0), \quad b_1(x)\leq C_N,\\
\forall & x \in [0,1],  \quad b_1(x)\leq \frac{1}{2}\left(\frac{C_N}{4}+1\right) ,\\
\forall & x \geq 1,  \quad b_1(x)\leq \frac{-1}{NZ^-}\left(\frac{C_N}{4}+1\right) .\\
\end{align*}
So finally $$L_N= Z^-\,\underset{x\in \R}{\sup}\left( b_1(x) + b_2(x) \right) \leq \max \left(\frac{1}{|N|}\left(\frac{C_N}{4}+2\right)\,,\,Z^-\,C_N+\frac{Z^-}{2}\left(\frac{C_N}{4}+1\right) \right),$$  which boils down to the result since $Z^-\leq C_N\leq \frac{2N}{N+1}$.
\end{proof}

\textbf{Acknowledgments: } This work was supported by the French ANR-17-CE40-0030 EFI project. M.F. and J.S were also supported by the ANR-18-CE40-0006 MESA project. We thank Jer\^ome Bertrand, Alexandros Eskenazis and Michel Ledoux for useful discussions.

\medskip

M. F., Laboratoire Jacques Louis Lions \& Laboratoire de Probabilit\'es Statistique et Mod\'elisation;
Universit\'e de Paris, France 	\texttt{mfathi@lpsm.paris}

	\medskip

I. G., {Institut Camille Jordan, Umr Cnrs 52065, Universit\'e Claude Bernard Lyon 1, 43 boulevard du 11 novembre 1918, F-69622 Villeurbanne cedex.
	\texttt{gentil@math.univ-lyon1.fr}
	
	\medskip
	
J. S., Institut de Math\'ematiques de Toulouse (UMR 5219). University of Toulouse. UPS, F-31062 Toulouse. 
	\texttt{jordan.serres@math.univ-toulouse.fr}

\end{document}